\documentclass[12pt]{amsart}

\usepackage{cite}

\usepackage[foot]{amsaddr}

\usepackage{kotex}
\usepackage{lineno,hyperref}
\usepackage{epstopdf}
\usepackage{pifont}
\usepackage{dsfont}
\usepackage{tikz}
\usepackage{latexsym}
\usepackage{graphics}
\usepackage{graphicx}
\usepackage{float}
\usepackage[dvips]{xy}
\xyoption{all}
\usepackage{ifpdf}
\usepackage{multirow}
\usepackage{amssymb, amsthm, amsmath}
\usepackage{hhline}
\usepackage{centernot}
\usepackage{verbatim, comment, color}
\usepackage{enumerate}
\usepackage{setspace}
\usepackage{lipsum}
\usepackage{subcaption}
\usepackage{mathtools}

\usepackage[numbers,sort&compress]{natbib}
\usepackage{enumitem}

\newtheorem{theorem}{Theorem} 
\newtheorem{proposition}[theorem]{Proposition}
\newtheorem{lemma}[theorem]{Lemma}
\newtheorem{remark}[theorem]{Remark}

\newtheorem{example}[theorem]{Example}

\newcommand{\ba}{\begin{align}}
\newcommand{\ea}{\end{align}}  

\newcommand{\be}{\begin{equation}}
\newcommand{\ee}{\end{equation}}

\newcommand{\bea}{\begin{eqnarray}}
\newcommand{\eea}{\end{eqnarray}}
\newcommand{\barr}{\begin{array}}
\newcommand{\earr}{\end{array}}
\newcommand{\bn}{\begin{enumerate}}
\newcommand{\en}{\end{enumerate}}
\newcommand{\bi}{\begin{itemize}}
\newcommand{\ei}{\end{itemize}}
\newcommand{\bbbm}{\begin{pmatrix}}
\newcommand{\eeem}{\end{pmatrix}}

\newcommand{\cF}{{\cal F}}
\newcommand{\cG}{{\cal G}}

\newcommand{\cN}{{\cal N}}
\newcommand{\cP}{{\cal P}}

\newcommand{\E}{{\mathbb E}}

\newcommand{\one}{{\mathds{1}}}
\newcommand{\N}{{\mathbb N}}

\newcommand{\R}{{\mathbb R}}

\newcommand{\al}{\alpha}

\newcommand{\ga}{\gamma}

\newcommand{\de}{\delta}

\newcommand{\la}{\lambda}
\newcommand{\La}{\Lambda}

\newcommand{\ignore}[1]{}{}
\newcommand{\noin}{\noindent}

\newcommand{\nn}{\nonumber}

\newcommand{\q}{\quad}

\newcommand{{\QED}}{{\hfill QED} \bigskip}

\renewcommand{\subset}{\subseteq}

\newcommand{\cal}{\mathcal}

\newcommand{\es}{\emptyset}

\definecolor{darkspringgreen}{rgb}{0.09, 0.45, 0.27} 
\definecolor{darkgray}{rgb}{0.66, 0.66, 0.66}

\numberwithin{equation}{section}
\numberwithin{theorem}{section}
\tikzset{
    partial ellipse/.style args={#1:#2:#3}{
        insert path={+ (#1:#3) arc (#1:#2:#3)}
    }
}

\begin{document}
\title[Hodge allocation for generalized cooperative games]
{Hodge theoretic reward allocation for generalized cooperative games on graphs}
\thanks{\em \copyright 2021 by the author. }

\date{\today}

\author{Tongseok Lim}
\address{Tongseok Lim: Krannert School of Management \newline  Purdue University, West Lafayette, Indiana 47907, USA}
\email{lim336@purdue.edu}

\begin{abstract} 

This paper generalizes L.S. Shapley's celebrated value allocation theory on coalition games by discovering and applying a fundamental connection between stochastic path integration driven by canonical time-reversible Markov chains and Hodge-theoretic discrete Poisson's equations on general weighted graphs.

More precisely, we begin by defining cooperative games on general graphs and generalize Shapley's value allocation formula for those games in terms of  stochastic path integral driven by the associated canonical Markov chain. We then show the value allocation operator, one for each player defined by the path integral, turns out to be the solution to the Poisson's equation defined via the combinatorial Hodge decomposition on general weighted graphs. 

Several motivational examples and applications are presented, in particular, a section is devoted to reinterpret and extend Nash's and Kohlberg and Neyman's solution concept for cooperative games. This and other examples, e.g. on revenue management, suggest that our general framework does not have to be restricted to cooperative games setup, but may apply to broader range of problems arising in economics, finance and other social and physical sciences.

\end{abstract}

\maketitle



\tableofcontents

\section{Introduction}
Let $\N$ denote the set of positive integers. For $N \in \N$, we let $[N]:=\{1,2,...,N\}$ denote the set of players. Let $\Xi$ be an arbitrary finite set which represents all possible cooperation states. The typical example is the choice $\Xi := 2^{[N]}$ in the classical work of Shapley \cite{Shapley1953a, Shapley1953}, where each $S \subset [N]$ represents the players involved in the {\em coalition} $S$. 

In this paper, each $S \in \Xi$, for instance, might contain more (or less) information than merely the list of players involved in the cooperation $S$, and this motivates to consider an abstract state space $\Xi$. We assume the null cooperation, denoted by $\emptyset$, is in $\Xi$; see examples in section \eqref{examples}. Now the set of \emph{cooperative games} is defined by 
\be
\cG = \cG(\Xi) :=  \{ v : \Xi \to \R \ | \ v(\emptyset)=0 \}. \nn
\ee
Thus a cooperative game $v$ assigns a value $v(S)$ for each cooperation $S$, where the null coalition $\emptyset$ is assigned zero value. For instance, $S, T \in \Xi$ could both represent the cooperations among the same group of players but working under different conditions, possibly yielding $v(S) \neq v(T)$.

 When $\Xi = 2^{[N]}$, L.~Shapley considered the question of how to split the grand coalition value $v([N]) $ among the players for each game $v \in \cG(2^{[N]})$. It is uniquely determined according to the following theorem.

\begin{theorem}[\citet{Shapley1953}]
  \label{thm:shapley}
There exists a unique allocation
  $ v \in \cG(2^{[N]}) \mapsto \bigl( \phi _i (v) \bigr) _{ i \in [N] } $ satisfying the following conditions:
  
\noin{\rm (i)} $ \sum _{ i \in [N] } \phi _i (v) = v ([N]) $.

\noin{\rm (ii)} If
    $ v \bigl( S \cup \{ i \} \bigr) = v \bigl( S \cup \{ j \} \bigr)
    $ for all $ S \subset [N] \setminus \{ i, j \} $, then
    $ \phi _i (v) = \phi _j (v) $.

\noin{\rm(iii)} 
    If $ v \bigl( S \cup \{ i \} \bigr) - v (S) = 0 $ for all
    $ S \subset [N] \setminus \{ i \} $, then $ \phi _i (v) = 0 $.

\noin{\rm (iv)} 
    $ \phi _i ( \alpha v + \alpha ^\prime v ^\prime ) = \alpha \phi _i
    (v) + \alpha ^\prime \phi _i ( v ^\prime ) $ for all
    $ \alpha, \alpha ^\prime \in \mathbb{R} $, $v,v' \in \cG(2^{[N]})$.

  Moreover, this allocation is given by the following explicit formula:
  \begin{equation}
    \label{eqn:shapley}
    \phi _i (v) = \sum _{ S \subset [N] \setminus \{ i \} } \frac{ \lvert S \rvert ! \bigl( N  - 1 - \lvert S \rvert \bigr) ! }{ N  ! } \Bigl( v \bigl(  S \cup \{ i \} \bigr) - v (S) \Bigr) .
  \end{equation}
\end{theorem}

The four conditions listed above are often called the
\emph{Shapley axioms}. Quoted from \cite{StTe2019}, they say that [(i)\,efficiency] the value obtained by the grand coalition is fully distributed among the players, [(ii)\,symmetry] equivalent players receive equal amounts, [(iii)\,null-player] a player who contributes no marginal value to any coalition receives nothing, and [(iv)\,linearity] the allocation is linear in the game values.

 \eqref{eqn:shapley} can be rewritten also quoted from \cite{StTe2019}:  Suppose the players form the grand coalition by
joining, one-at-a-time, in the order defined by a permutation $\sigma$
of $[N]$. That is, player $i$ joins immediately after the coalition
$ S _{ \sigma, i } = \bigl\{ j \in [N] : \sigma (j) < \sigma (i) \bigr\}
$ has formed, contributing marginal value
$ v \bigl( S _{ \sigma, i } \cup \{ i \} \bigr) - v (S _{ \sigma, i }
) $. Then $ \phi _i (v) $ is the average marginal value contributed by
player $i$ over all $ N ! $ permutations $\sigma$, i.e.,
\begin{equation}
  \label{eqn:shapleyPermutation}
  \phi _i (v) = \frac{ 1 }{ N ! } \sum _\sigma \Bigl( v \bigl( S _{ \sigma, i } \cup \{ i \} \bigr) - v (S _{ \sigma, i } ) \Bigr) .
\end{equation} 
Here we notice an important principle, which we may call {\em Shapley's principle}, which says the value allocated to player $ i  $ is based entirely on the marginal values
$ v \bigl( S \cup \{ i \} \bigr) - v (S) $ the player $i$ contribute.

The pioneering study of Shapley \cite{Shapley1953a, Shapley1953, Shapley1953c, Shapley2010} have been followed by many researchers with extensive and diverse literature. For instance, \citet {Young1985} and \citet{Chun1989} studied Shapley's axioms and suggested its variants. \citet{Roth1977} studied the requirement of the utility function for games under which it is unique and equal to the Shapley value. \citet{Gul1989}  studied the relationship between the cooperative and noncooperative approaches by establishing a framework in which the results of the two theories can be compared. We refer to  \citet{Roth1988} and \citet{PeSu2007} for more detailed exposition of cooperative game theory.

More recently, the combinatorial Hodge decomposition has been applied to game theory and  various economic contexts, for instance \citet{CaMeOzPa2011}, \citet{JiLiYaYe2011}, \citet{StTe2019}.  We refer to \citet{Lim2020} for an accessible introduction to the Hodge theory on graphs.  Another important direction we note is the {\em mean field game theory}, the study of strategic decision making by interacting agents in very large populations; see \citet{CaDeLaLi2019}, \citet{AcBaCa2019}, \citet{BaCvZh2019, BaCoCeDe2021}, \citet{PoToZh2020}, \citet{LaSo2021}, for instance.

In particular,  \citet{StTe2019} showed that, given a game $v \in \cG(2^{[N]})$, there exist {\em component games} $v_i \in \cG(2^{[N]})$ for each player $i \in [N]$ which are naturally defined via the {\em combinatorial Hodge decomposition}, satisfying $v= \sum_{i \in [N]} v_i$. Moreover, they showed 
\be\label{shapleyrecover}
 v_i([N]) = \phi _i (v) \text{ \ for every \ } i \in [N]
 \ee
 hence they obtained a new characterization of the Shapley value as the value of the grand coalition in each player’s component game.

In this context, the {\em combinatorial Hodge decomposition} corresponds to the elementary {\em Fundamental Theorem of Linear Algebra}. For finite-dimensional inner product spaces $X, Y$ and a linear map $\mathrm{d} : X \to Y$ and its adjoint $\mathrm{d}^* : Y \to X$ given by $  \langle \mathrm{d}x , y \rangle_Y =   \langle x , \mathrm{d}^*y \rangle_X$, FTLA asserts that the  orthogonal decompositions hold:
\be\label{Hodge}
X = \mathcal{R} ( \mathrm{d} ^\ast ) \oplus \mathcal{N} ( \mathrm{d} ) , \qquad Y = \mathcal{R} ( \mathrm{d} ) \oplus \mathcal{N} ( \mathrm{d} ^\ast ),
\ee
where $\cal R(\cdot)$, $\cN(\cdot)$ stand for the range and nullspace respectively.

In order to introduce the work of \cite{StTe2019} and \cite{Lim2021}, let us briefly review their setup. Let $ G = ( V, E ) $ be an oriented graph, where $V$ is the set of
vertices and $E \subset V \times V $ is the set of edges. ``Oriented" means at most one of $ ( a, b ) $ and
$ ( b, a ) $ is in $E$ for $ a, b \in V $. Let $ \ell ^2 (V) $ be the space of functions
$ V \rightarrow \mathbb{R} $ with the (unweighted)  inner
product
\begin{equation}\label{noinner1}
  \langle u , v \rangle \coloneqq \sum _{ a \in V } u (a) v (a).
\end{equation}
Denote by $ \ell ^2 (E) $ the space of functions
$E \rightarrow \mathbb{R} $ with inner product
\begin{equation}\label{noinner2}
  \langle f, g \rangle \coloneqq \sum _{ (a,b) \in E } f (a,b) g (a,b)
\end{equation}
with the convention that, if $ f \in \ell ^2 (E)  $ and $ (a,b) \in E $, we define
$ f(b,a) \coloneqq - f (a,b) $ for the reverse-oriented edge. Thus every $ f \in \ell ^2 (E)  $ is defined on all the edges in $E$ and their reverse.

Next, define a linear operator
$ \mathrm{d} \colon \ell ^2 (V) \rightarrow \ell ^2 (E) $, the {\em gradient}, by
\begin{equation}\label{gradient}
  \mathrm{d} u ( a, b ) \coloneqq u (b) - u (a).
\end{equation}
Its adjoint $ \mathrm{d} ^\ast \colon \ell ^2 (E) \rightarrow \ell ^2 (V) $, the (negative) {\em divergence}, is then
\begin{equation}\label{divergence}
  (\mathrm{d} ^\ast f) (a) = \sum _{ b \sim a } f(b,a),
\end{equation} 
where $ b \sim a $ denotes $ (a,b) \in E $ or $ ( b, a ) \in E $, i.e., $a,b$ are adjacent.

 Now to study the cooperative games,  \citet{StTe2019} applied the above setup to the \emph{hypercube graph} $ G = ( V, E ) $, where
\begin{equation}\label{oldG}
  V = 2 ^{[N]}, \ \  E = \bigl\{ \bigl(  S, S \cup \{ i \} \bigr) \in V \times V \ | \  S \subset [N] \setminus \{ i \} ,\ i \in [N] \bigr\} .
\end{equation}
Note that each vertex $S \subset [N]$ may correspond to a vertex of the unit hypercube in $\R^{N}$, and each edge is oriented in the direction of the inclusion $ S \hookrightarrow S \cup \{ i \} $. Then for each $ i \in [N] $, \cite{StTe2019} set
  $ \mathrm{d} _i \colon \ell ^2 (V) \rightarrow \ell ^2 (E) $ as the following {\em partial differential operator}
  \begin{equation}\label{oldpartial}
    \mathrm{d} _i u \bigl( S , S \cup \{ j \} \bigr) =
    \begin{cases}
      \mathrm{d} u \bigl( S , S \cup \{ i \} \bigr) & \text{if } \ j=i, \\
      0 & \text{if } \  j \neq i.
    \end{cases}
  \end{equation} 
Thus $ \mathrm{d} _i v \in \ell ^2 (E) $ encodes the marginal
value contributed by player $i$ to the game $v$, which is a natural object to consider in view of the Shapley's principle. Indeed, for $v \in \cG(2^{[N]})$, \citet{StTe2019} defined the component game $v_i$ for each $i \in [N]$ as the unique solution in $\cG(2^{[N]})$ to the following {\em least squares}, or {\em Poisson's}, {\em equation}\footnote{The equation $\mathrm{d} u = f$  is solvable if only if $f \in \mathcal{R} (\mathrm{d} )$. When $f \notin \mathcal{R} (\mathrm{d} )$, a least squares solution to $\mathrm{d}u=f$ instead solves $\mathrm{d}u=f_1$ where $f=f_1+f_2$ with $f_1 \in \mathcal{R} (\mathrm{d} )$, $f_2 \in  \mathcal{N} (\mathrm{d}^\ast )$ given by FTLA. By applying $\mathrm{d}^\ast$, we get $\mathrm{d}^\ast \mathrm{d} u = \mathrm{d}^\ast f_1 = \mathrm{d}^\ast f$. Here the substitution $u \to v_i$ and $f \to \mathrm{d}_i v$ yields \eqref{ls}. Note the equation $\mathrm{d} ^\ast \mathrm{d} v = \mathrm{d} ^\ast f$ may be called a Poisson's equation since $\mathrm{d} ^\ast \mathrm{d}$ is the Laplacian and $\mathrm{d} ^\ast$ is the divergence.
}
\be\label{ls}
\mathrm{d} ^\ast \mathrm{d} v_i = \mathrm{d} ^\ast \mathrm{d}_i v
\ee
and showed that the component games satisfy some natural properties analogous to the Shapley axioms (see \cite[Theorem 3.4]{StTe2019}). Moreover, by applying the inverse of the {\em Laplacian} $\mathrm{d} ^\ast \mathrm{d}$ to \eqref{ls}, they provided an explicit formula for $v_i$ (see \cite[Theorem 3.11]{StTe2019}). In addition, \cite{StTe2019} discussed the case of weighted hypercube graph, viewing this as modeling variable willingness or unwillingness of players to join certain coalitions. More recently \citet{Lim2021}, inspired by \citet{StTe2019}, proposed a generalization of the Shapley axioms and showed that the extended axioms completely characterize the component games $(v_i)_{i \in [N]}$ defined by \eqref{ls} for the unweighted hypercube graph. 

Now the first goal of this paper is to generalize Shapley's coalition space $2^{[N]}$ into general cooperative state space $\Xi$. For this we consider directed graphs $G=(V,E)$ with $V= \Xi$, which can now be weighted. For each weighted graph $G$, we then associate a canonical Markov chain whose transition rates model the probability of which direction the cooperation would progress toward. Then powered by this Markov chain we introduce our main objective of study, the value function $V_i \in \cG(\Xi)$ for each player $i \in [N]$, described by a stochastic path integral such that $V_i(S)$ represents the expected total contribution the player $i$ provides toward each cooperation $S$. This may be viewed as a generalization of the Shapley formula for the cooperative games defined on the abstract cooperation network $G=(\Xi, E)$. Finally, our main result reveals  the stochastic integral $V_i$ is in fact the solution to Poisson's equation \eqref{ls2}, and therefore the value functions $(V_i)_{i \in [N]}$ coincide with the component games $(v_i)_{i \in [N]}$ which are defined via the equation \eqref{ls1}. As a result, this would justify the interpretation of the component game value $v_i(S)$ to be a reasonable reward allocation for player $i$ at the cooperation $S$. 

To the best of the author's knowledge, such a stochastic integral representation of an allocation scheme, and its connection to the Poisson's equation on general graphs, has not been discussed in the literature, and we hope our analysis and interpretation of this interesting connection  will eventually open up new directions in many scientific domains. 

This paper is organized as follows. In section \ref{connection} we introduce the reward allocation function for each player via a stochastic path integral driven by a Markov chain on a given graph, and verify its connection with the discrete Poisson's equation on the graph. In section \ref{Nash}, we present a dynamic interpretation and extension of Nash's and Kohlberg and Neyman's solution concept for strategic games to demonstrate the relevance of our probabilistic value allocation  with existing literature. In section \ref{examples}, we provide additional motivation by illustrating the generalized concepts  proposed in this paper through examples.

\section{Component game, path integral representation of  reward allocation, and their coincidence}\label{connection}

We begin by defining  the inner product space of functions $ \ell ^2 (\Xi) $, $ \ell ^2 (E) $, now possibly weighted. That is, let $\mu$, $\la$ be strictly positive weight functions on $\Xi$ and $E$ respectively, and set $\la (T,S) = \la (S,T)$ for any $(S,T) \in E$ by convention.  
Denote by $\ell^2_\mu(\Xi)$ the space of functions  $ V \rightarrow \mathbb{R} $ equipped with the ($\mu$-weighted) inner product
\begin{equation}\label{inner1}
   \langle u , v \rangle_{\mu} \coloneqq \sum _{ S \in \Xi } \mu (S) u (S) v (S).
\end{equation}
Denote by $ \ell ^2_\la (E) $ the space of functions
$E \rightarrow \mathbb{R} $ with inner product
\begin{equation}\label{inner2}
   \langle f , g \rangle_{\la} \coloneqq \sum _{ (S,T) \in E } \la (S,T) f (S,T) g (S,T)
\end{equation}
with the convention 
$ f(T,S) \coloneqq - f (S,T) $ for the reverse-oriented edge. We would say for $S,T \in \Xi$, there exists a (forward- or reverse-oriented) edge $(S,T)$ if and only if $\la(S,T) >0$. Then we say the weighted graph $(G, \la) = ((\Xi, E), \la)$ is {\em connected} if for any $S, T \in \Xi$ there exists a chain of (forward- or reverse-) edges $\big((S_k, S_{k+1}) \big)_{k=0}^{n-1}$ with $S_0 = S$, $S_n = T$. We assume $\emptyset \in \Xi$, so  every $S \in \Xi$ is connected with $\emptyset$, for convenience.

\subsection{Component games for cooperative game on general graph}
Recall the linear map, gradient, $\mathrm{d}: \ell^2_\mu(\Xi) \to \ell^2_\la(E)$ \eqref{gradient} between the inner product spaces.  We have an adjoint (divergence) $\mathrm{d}^\ast$, given by
\be\label{adjoint}
 \langle \mathrm{d} v, f \rangle_{\la} =  \langle v , \mathrm{d}^\ast f \rangle_{\mu}.
\ee
It is not hard to find the explicit form of $\mathrm{d}^*$. Let $(\one_S)_{S \in \Xi}$ be the standard basis of $\ell^2(\Xi)$, where $\one_S (T)  =1 $ if $T = S$ and otherwise $0$. Then
\begin{align}\label{div}
\mathrm{d}^*f (S) = \frac{ \langle \one_S , \mathrm{d}^\ast f \rangle_{\mu} }{\mu(S)}
= \frac{ \langle \mathrm{d} \one_S , f \rangle_{\la} }{\mu(S)}
=  \sum_{T \sim S} \frac{\la (T,S)}{\mu(S)}  f(T,S). 
\end{align}
Next we recall the partial differential operator $\mathrm{d}_i$ in \eqref{oldpartial}. While this is a natural definition for a measure of the contribution of player $i$ in the hypercube  graph setup \eqref{oldG}, it does not seem to readily apply for our general graph $G$. But the observation here is that $\mathrm{d}_i$ may not have to be a linear operator acting on the game space $\cG$. Instead, we can be utterly general and  define each player's contribution to be an arbitrary element in $\ell^2 (E)$. That is, let $\vec f =(f_1,...,f_N) \in \otimes_{i=1}^N \ell^2(E)$
 denote the $N$-tuple of player contribution measures, where $f_i(S,T)$ indicates player $i$'s contribution when the cooperation proceeds from $S$ to $T$.

Given $\vec f$,  we define the {\em component game} $v_i \in \cG(\Xi)$, for each player $i$, by the solution to the least squares / Poisson's equation (cf. \eqref{ls})
\be\label{ls1}
\mathrm{d} ^\ast \mathrm{d} v_i = \mathrm{d} ^\ast f_i.
\ee
Given an initial condition, \eqref{ls1} admits a unique solution so $v_i$ is well defined. This is because $G$ is connected and thus $ \cal N(\mathrm{d})$ is one-dimensional space spanned by the constant game $\one$, defined by $\one (S) := 1 $ for all $S \in \Xi$. Hence if $\mathrm{d} ^\ast \mathrm{d} v_i  = \mathrm{d} ^\ast \mathrm{d} w_i $, then $v_i - w_i \in \cal N(\mathrm{d})$ but due to the initial condition $v_i (\es) = w_i (\es) =0 $ from the assumption $v_i, w_i \in \cG(\Xi)$, we have $v_i \equiv w_i$. This is the reason we assume the connectedness of $G$.

But note that what \eqref{ls1} actually determines is the increment $d v_i$ in each connected component of $G$. Thus by assigning an initial value $v_i(S)$ for some $S$ in each connected component,  $v_i$ will be determined in that component via \eqref{ls1}. Here we shall assume, without loss of generality, $G$ is connected with initial condition $v_i(\es) =0$ for all $i$. But this is not strictly necessary, and e.g. one may assign any value for $v_i(\es)$ for each $i$, thereby modeling some sort of inequality at the initial stage.

Let us gather some results regarding the component games, whose proof is analogous to \citet{StTe2019} and \citet{Lim2021}.

\begin{proposition}
\label{facts} 
Given $(v, (f_i)_i, \mu, \la)$ consisting of the cooperative game, contribution measures and graph weights, the component games $(v_i)_{i \in [N]}$ defined by \eqref{ls1} satisfy the following:

{\rm {\boldmath$\cdot$} efficiency:} If $\sum_i f_i = \mathrm{d}v$, then $\sum _{ i \in [N] } v_i = v $.

{\rm {\boldmath$\cdot$} null-player:} 
    If $f_i \equiv 0$, then $v_i \equiv 0 $.

{\rm {\boldmath$\cdot$} linearity:} If we assume $f_i := \mathrm{d}_i v$ for a fixed linear map $\mathrm{d}_i : \ell^2_\mu(\Xi) \to \ell^2_\la(E)$, then 
$(\alpha v + \alpha ^\prime v ^\prime )_i = \alpha 
    v_i + \alpha ^\prime  v ^\prime_i $ for all
    $ \alpha, \alpha ^\prime \in \mathbb{R} $ and $v,v' \in \cG$.
 \end{proposition}
\begin{proof} The null-player  property is immediate from the defining equation \eqref{ls1} and the initial condition $v_i(\es)=0$. For efficiency, we compute
  \begin{equation*}
     \mathrm{d}^\ast \mathrm{d} \sum _{ i \in [N] } v _i = \sum _{ i \in [N]}   \mathrm{d}^\ast\mathrm{d} v _i = \sum _{ i \in [N] }   \mathrm{d}^\ast f_i =   \mathrm{d}^\ast \sum _{ i \in [N] } f_i =   \mathrm{d}^\ast \mathrm{d} v
  \end{equation*}
thus efficiency follows by the unique solvability of \eqref{ls1}. Finally,  linearity follows by the assumed linearity of the map $\mathrm{d}_i$:
   \begin{align*}
     \mathrm{d}^\ast \mathrm{d} (\alpha v + \alpha ^\prime v ^\prime )_i &=      \mathrm{d}^\ast \mathrm{d}_i (\alpha v + \alpha ^\prime v ^\prime ) 
     = \al    \mathrm{d}^\ast \mathrm{d}_i v + \al'    \mathrm{d}^\ast \mathrm{d}_i v' \\
     & = \al   \mathrm{d}^\ast \mathrm{d} v_i + \al'   \mathrm{d}^\ast \mathrm{d} v_i' =\mathrm{d}^\ast \mathrm{d} (\al v_i +\al'v'_i)
       \end{align*}
       yielding $(\alpha v + \alpha ^\prime v ^\prime )_i = \alpha 
    v_i + \alpha ^\prime  v ^\prime_i $ as desired.
\end{proof}

Note that the $\mathrm{d}_i$ given in \eqref{oldpartial} is an example of a linear map. Also notice that we do not present a symmetry property analogous to the Shapley axiom \ref{thm:shapley}(ii), due to the fact that unlike the hypercube graph \eqref{oldG}, a general graph $G$ may not exhibit any obvious symmetry.

Next let us observe that, although the weight $\mu$ also affects the divergence $\mathrm{d}^*$ as in \eqref{div}, in fact  it does not affect the component games.
\begin{lemma}\label{noeffect}
Let $f \in \ell^2_\la(E)$. Then the solution $v \in \ell^2_\mu(\Xi)$ to the equation $\mathrm{d}^*\mathrm{d}v= \mathrm{d}^* f$ does not depend on the choice of $\mu$.
\end{lemma}
\begin{proof} 
$(\mathrm{d}^*\mathrm{d}v - \mathrm{d}^* f) (S) =  \frac{1}{\mu(S) }\sum_{T \sim S} \la (T,S) [ v(S) - v(T) -   f(T,S) ]$
 shows that $(\mathrm{d}^*\mathrm{d}v - \mathrm{d}^* f) (S) = 0$ if and only if $\sum_{T \sim S} \la (T,S) [ v(S) - v(T) -   f(T,S) ] =0$, showing there is no dependence on $\mu$.
\end{proof}
On the other hand, the solution to $\mathrm{d}^*\mathrm{d}v= \mathrm{d}^* f$ does depend on $\la$. \cite{StTe2019} demonstrates this with several explicit computations of component games for weighted and unweighted hypercube graph \eqref{oldG}.

\subsection{Value allocation operator via a stochastic path integral}
Now we define our main objective of study, the reward allocation function $V_i$ for each player $i$, via a stochastic path integral driven by a Markov chain which is naturally associated to the given weighted graph. 

Let $\N_0= \N \cup \{0\}$ and recall that $\la$ denotes the edge weight \eqref{inner2}. Given $\la$, let us consider the canonical Markov chain $(X^U_n)_{n \in \N_0}$ on the state space $\Xi$ with $X_0 = U$ (with the convention $X_n := X^\emptyset_n$), equipped with the transition probability $p_{S,T}$ from a state $S$ to $T$ as follows: 
 \begin{align}\label{MC}
p_{S,T} = \frac{\la (S,T)} { \sum_{U \sim S} \la (S,U)} \ \text{ if } \ T \sim S, \q  p_{S,T} = 0 \ \text{ if } \ T \not\sim S. 
\end{align} 
Notice the weight $\la$ determines which direction the cooperation is likely to progress. This allows us further flexibility for modeling stochastic cooperation network. Also, we remark that our framework can apply if one can summarize his/her cooperative project into a graph and boil down related strategic/stochastic ingredients into the probability of state progression, which is described by the graph weight $\la$ and  \eqref{MC}.

It turns out that the Markov chain \eqref{MC} is {\em time-reversible}, meaning that there exists the stationary distribution $\pi=(\pi_S)_{S \in \Xi}$ such that 
\be
\pi_S p_{S,T} = \pi_T p_{T,S} \ \text{ for all } \ S, T \in \Xi.
\ee
A consequence, which is important to us, is that every loop and its reverse have the same probability, that is (see, e.g.,  \citet{Ross2019})
\be\label{reversibility}
p_{S,S_1} p_{S_1,S_2} \dots  p_{S_{n-1},S_n} p_{S_n, S} =p_{S,S_n} p_{S_n,S_{n-1}} \dots  p_{S_2,S_1} p_{S_1, S}.
\ee
Let $(\Omega, \cF, \cP)$ be the underlying probability space for the Markov chain. For each $S,T \in \Xi$ and $\omega \in \Omega$, let $\tau_{S,T} = \tau_{S,T}(\omega) \in \N_0$ denote the first (random) time the Markov chain $\big(X^S_n(\omega)\big)_n$ visits $T$. Given a player's contribution measure $f \in \ell^2(E)$, we define the total contribution of the player along the sample path $\omega \in \Omega$ traveling from $S$ to $T$ by
\be\label{pathintegral}
{\cal I}_f^S(T) = {\cal I}_f^S(T)(\omega) := \sum_{n=1}^{\tau_{S,T}(\omega)} f  \big(X^S_{n-1}(\omega), X^S_n(\omega) \big).
\ee
Now we can define the value function for given $f \in \ell^2(E)$ via the following stochastic path integral driven by the Markov chain \eqref{MC} 
\be\label{value}
V_f^S (T) := \int_\Omega  {\cal I}_f^S(T)(\omega)  d \cP(\omega) = \E[ {\cal I}_f^S(T)]. 
\ee
Finally, let us denote $V_i^S := V_{f_i}^S$ for each player $i \in [N]$ given the players' contribution measures $(f_i)_{i \in [N]}$. One may notice that this path integral representation can be seen as a generalization of the Shapley formula \eqref{eqn:shapleyPermutation}. In particular, $V_i(T):= V_i^\emptyset(T)$ represents the expected total contribution the player $i$ provides toward each cooperation $T$, provided the game starts at the null cooperation state $\emptyset$.

\subsection{The coincidence between the value allocation operator and the component game}
The question is how we can compute the value allocators $(V_i)_{i \in [N]}$ which are described by the stochastic path integral. One could employ some computational methods to simulate the Markov chain and approximate the path integral, for instance.

Or, better yet, our main result of this paper shows that $V_i$ is a valid representation of the component game $v_i$, that is, $V_i = v_i$ for every player $i$. This result displays a remarkable connection between stochastic path integrals and combinatorial Hodge theory on general graphs.

First, we need to establish a transition formula for the value function. Note that in the proofs, we implicitly use the fact that the Markov chain is irreducible and hence visits every state infinitely many times. 
\begin{lemma}\label{transition}
Let $(G, \la)$ be any connected weighted graph. For any $S,T, U \in \Xi$ and $f \in \ell^2(E)$, we have $ V^U_f(T) - V^U_f(S) = V_f^S(T)$.
\end{lemma}
 \begin{proof}
We first prove a special case $V_f^S(T) = - V_f^T(S)$. Consider a general sample path $\omega$ of the Markov chain \eqref{MC} starting at $S$, visiting $T$, then returning to $S$. We can split this journey into four stages:\\
$\omega_1$: the path returns to $S$ $m \in \N_0$ times while not visiting $T$ yet,\\
$\omega_2$: the path starts at $S$ and ends at $T$ while not returning to $S$,\\
$\omega_3$: the path returns to $T$ $n \in \N_0$ times while not visiting $S$ yet,\\
$\omega_4$: the path starts at $T$ and ends at $S$ while not returning to $T$.\\
Thus $\omega = \omega_1 \circ \omega_2  \circ \omega_3 \circ \omega_4$ is the concatenation of the $\omega_i$'s, and the probability of this finite sample path satisfies $\cP(\omega) = \cP(\omega_1)\cP(\omega_2)\cP(\omega_3)\cP(\omega_4)$.

Now consider a pairing $\omega'$ of $\omega$ as follows: let $\omega^{-1}_1$ be the reversed path of $\omega_1$, that is, if $\omega_1$ visits $T_0 \to T_1 \to \dots \to T_k$ (where $T_0 = T_k = S$ for $\omega_1$), then $\omega^{-1}_1$ visits $T_k \to \dots \to T_0$. Recall $\cP(\omega_1) = \cP(\omega^{-1}_1)$ due to the time-reversibility \eqref{reversibility}. Now define $\omega' := \omega^{-1}_1  \circ \omega_2  \circ \omega_3^{-1}  \circ \omega_4$. This is another general sample path starting at $S$, visiting $T$, then returning to $S$. We have $\cP(\omega) = \cP(\omega')$, and moreover,
\be
{\cal I}_f^S(T)(\omega) + {\cal I}_f^S(T)(\omega') = 2 \sum_{n=1}^{\tau_{S,T}(\omega_2)} f  \big(X^S_{n-1}(\omega_2), X^S_n(\omega_2) \big), \nn
\ee
since the loop $\omega_1$ and its reverse $\omega^{-1}_1$ aggregate $f$ with opposite sign, so they cancel out in the above sum. Now consider $\tilde \omega := \omega_3 \circ \omega_2^{-1} \circ \omega_1 \circ \omega_4^{-1}$ and $\tilde \omega' := \omega_3^{-1} \circ \omega_2^{-1} \circ \omega_1^{-1} \circ \omega_4^{-1}$. $(\tilde \omega, \tilde \omega')$ then represents a pair of general sample paths starting at $T$, visiting $S$, then returning to $T$. Moreover
\begin{align}
{\cal I}_f^T(S)(\tilde \omega) + {\cal I}_f^T(S)(\tilde \omega') &= 2 \sum_{n=1}^{\tau_{T,S}(\omega_2^{-1})} f  \big(X^T_{n-1}(\omega_2^{-1}), X^T_n(\omega_2^{-1}) \big) \nn \\
&= -2 \sum_{n=1}^{\tau_{S,T}(\omega_2)} f \big(X^S_{n-1}(\omega_2), X^S_n(\omega_2) \big) \nn \\
&= - ({\cal I}_f^S(T)(\omega) + {\cal I}_f^S(T)(\omega')) \nn
\end{align}
since $f \in \ell^2(E)$. Due to the generality of the pair $(\omega, \omega')$ and its counterpart $(\tilde \omega, \tilde \omega')$, and $\cP(\omega) = \cP(\omega') = \cP(\tilde \omega) = \cP(\tilde \omega')$ from the reversibility \eqref{reversibility}, the desired identity $V_f^S(T) = - V_f^T(S)$ follows by integration.

Now to show $ V_f^U(T) - V_f^U (S) = V^S_f(T)$, we proceed as in \cite{Lim2021}:
\begin{align*}
&{\cal I}_f^U(T) - {\cal I}_f^U (S) = \sum_{n=1}^{\tau_{U,T}} f \big( X^U_{n-1}, X^U_n \big) 
-  \sum_{n=1}^{\tau_{U,S}} f  \big( X^U_{n-1}, X^U_n \big) \\
&= {\bf 1}_{\tau_{U,S} < \tau_{U,T}}\sum_{n=\tau_{U,S} + 1}^{\tau_{U,T}} f  \big( X^U_{n-1}, X^U_n \big) 
- {\bf 1}_{\tau_{U,T} < \tau_{U,S}}\sum_{n=\tau_{U,T} + 1}^{\tau_{U,S}} f \big( X^U_{n-1}, X^U_n \big).
\end{align*}
By taking expectation, we obtain via the Markov property
\begin{align*}
&\E[{\cal I}_f^U(T)]  - \E[{\cal I}_f^U (S)] \\
 &= \cP(\{\tau_{U,S} < \tau_{U,T} \}) V^S_f(T) 
- \cP(\{\tau_{U,T} < \tau_{U,S}\}) V^T_f(S) \\
&= V^S_f(T)
\end{align*}
which proves the transition formula  $ V_f^U(T) - V_f^U (S) = V^S_f(T)$.
\end{proof}
Now we present our main result. 
\begin{theorem}\label{main1}
Let $f \in \ell^2(E)$ and let the Markov chain \eqref{MC} be defined on a weighted graph $(G, \la)$. Then $V^S_f$ solves the Poisson's equation
\be\label{ls2}
\mathrm{d}^*\mathrm{d}V^S_f= \mathrm{d}^* f 
\ee
on the connected component of $G$ to which the state $S$ belongs.
\end{theorem}
The theorem tells us when one wants to calculate the value allocation function $V_i$ for the player $i$ given her contribution measure $f_i$, one can instead compute the least squares solution $v_i$, which can be easily done via least squares solvers for instance. Conversely, the least squares solution $v_i$ may be approximated by simulating the canonical Markov chain \eqref{MC} on the graph $(G, \la)$ and calculating the contribution aggregator \eqref{value}. Both directions look interesting and potentially useful. 
\begin{proof}[Proof of Theorem \ref{main1}.] 
Recall the weight $\mu$ on $\ell^2(\Xi)$ is not relevant in either Lemma \ref{noeffect} or \eqref{MC}, so we will simply set $\mu \equiv 1$. 
Given $f \in \ell^2(E)$, our aim is to show that $V_f$ solves \eqref{ls2}.  Let $S \in \Xi$, and let $\{T_1,...,T_n\}$ be the set of all vertices adjacent to $S$ (i.e., either $(S, T_k)$ or $(T_k,S)$ is in $E$), and set $\La_S = \sum_{k=1}^n \la (S, T_k)$.
  Then by \eqref{div}, \eqref{MC}, we have
   \begin{align}
 \label{divf}
\mathrm{d}^*f (S) / \La_S &= \sum_{k=1}^n p_{S, T_k}  f(T_k,S), \text{ and} \\
\label{divV}
\mathrm{d}^* \mathrm{d} V_f (S) / \La_S &=  \sum_{k=1}^n p_{S, T_k}\big(V_f(S) - V_f(T_k)\big) =  \sum_{k=1}^n p_{S, T_k} V_f^{T_k} ( S) 
\end{align}
where the last equality is from Lemma \ref{transition}. Now observe that we can interpret \eqref{divV} as the aggregation \eqref{value} of path integrals of $f$ \eqref{pathintegral} for all loops starting and ending at $S$, but in this aggregation of $f$ we do not take into account the first move from $S$ to $T_k$, since this first move is made by the transition rate $p_{S, T_k}$ and not driven by $V^{T_k}_f$. On the other hand, if we aggregate path integrals of $f$ for all loops emanating from $S$, we get zero due to the reversibility \eqref{reversibility}. Hence we conclude:
\begin{align*}
&0 = \text{aggregation of path integrals of $f$ for all loops emanating from $S$} \\
&= \text{aggregation of path integrals of $f$ for all loops, omitting the first moves} \\
&\, + \text{aggregation of path integrals of $f$ for all first moves from $S$}\\
&= \sum_{k=1}^n p_{S, T_k} V_f^{T_k} ( S) +  \sum_{k=1}^n p_{S, T_k} f(S,T_k) \\
&= \mathrm{d}^* \mathrm{d} V_f (S) / \La_S - \mathrm{d}^*f (S) / \La_S,
\end{align*}
yielding $ \mathrm{d}^* \mathrm{d} V_f (S) = \mathrm{d}^*f (S)$ for all $S \in \Xi$, concluding the proof.
\end{proof}

\section{Dynamic interpretation and extension of Nash's and Kohlberg and Neyman's value allocation scheme
}\label{Nash}
Quoted from \citet{KoNe2021}, a {\em strategic game} is a model for a multiperson competitive interaction. Each player chooses a strategy, and the combined choices of all the players determine a payoff to each of them. A problem of interest in game theory is the following: How to evaluate, in advance of playing a game, the economic worth of a player’s position? A ``value" is a general solution, that is, a method for evaluating the worth of any player in a given strategic game.

In this section we briefly introduce Nash's and Kohlberg and Neyman's value, and explain how their axiomatic notion of value can be reinterpreted in terms of our dynamic value allocation operator, and as a consequence, can be extended to partial (i.e., non-grand) coalitions.

A strategic game in \cite{KoNe2021} is defined by a triple $G=([N],A,g)$, where

{\boldmath$\cdot$} $[N]=\{1,2,...,N\}$ is a finite set of players,

{\boldmath$\cdot$} $A^i$ is the finite set of player $i$’s pure strategies, and $ A = \prod_{i=1}^n A^i$,

{\boldmath$\cdot$} $g^i : A \to \R$ is player $i$’s payoff function, and $g = (g^i)_{i \in \N}$.

The same notation, $g$, is used to denote the linear extension

{\boldmath$\cdot$} $g^i : \Delta(A) \to \R$,

where for any set $K$, $\Delta(K)$ denotes the probability distributions on $K$. For each (partial) coalition $S \subset [N]$, we also denote

{\boldmath$\cdot$} $A^S = \prod_{i \in S} A^i$, and

{\boldmath$\cdot$} $X^S = \Delta(A^S)$ (correlated strategies of the players in $S$).

Denote by $\mathbb G([N])$ the set of all $N$-player strategic games, and consider $\ga : \mathbb G([N]) \to \R^N$ which may be viewed as a map that associates with any strategic game an allocation of payoffs to the players. Now Kohlberg and Neyman suggested a list of axioms for $\ga$, where the core notion is the following  definition of the {\em threat power} of coalition $S$:
\be\label{threat}
(\de G)(S) := \max_{x \in X^S} \min_{y \in X^{[N] \setminus S}} \bigg( \sum_{i \in S} g^i (x,y) - \sum_{i \notin S} g^i (x,y) \bigg).
\ee
Intuitively, the threat power of $S$ may read as the maximal difference of the sum of the players' payoffs in $S$ against the other party $[N]/S$, regardless of what collective strategies the other party implements.

Then Kohlberg and Neyman showed the axioms of {\em Efficiency} (the maximal sum of all players' payoffs, $\de G([N])$, is fully distributed among the players), {\em Balanced threats} (see below), {\em Symmetry} (equivalent players receive equal
amounts), {\em Null player} (a player having no strategic impact on players' payoffs has zero value), and {\em Additivity} (the allocation is additive on strategic games) uniquely determine an allocation $\ga$; see \cite{KoNe2021} for details. Moreover, such allocation $\ga$ is a  generalization  of the Nash solution for two-person games \cite{Nash1953} into $N$-person games.

Among the axioms, the axiom of balanced threats reads:
\vspace{2mm}

{\boldmath$\cdot$} If $(\delta G)(S) = 0$ for all $S \subset [N]$, then $\ga_i = 0$ for all $i \in [N]$.
\vspace{2mm}

In words, if no coalition has threat power over the other party, then the allocation is zero for all players. From now on let $\ga=(\ga_1,...,\ga_N)$ denote the unique allocation map determined by the above five axioms. Kohlberg and Neyman also provided an explicit formula for $\ga$ as
\be\label{formula}
\ga_i G = \frac{1}{N!} \sum_{\cal R} (\de G) \big(\bar S^{\cal R}_i \big),
\ee
where the summation is over the $N!$ possible orderings of the set $[N]$, $S^{\cal R}_i$ denotes the subset consisting of those $j \in [N]$ that precede $i$ in the ordering $\cal R$, and $\bar S^{\cal R}_i := S^{\cal R}_i \cup \{ i \}$.

Now let us slightly rewrite \eqref{formula} as follows. By minimax principle, it is easily seen that $\de G( S) = - \de G([N] \setminus S)$. This {\em antisymmetry} gives
\begin{align*}
\ga_i G &=  \frac{1}{N!} \sum_{\cal R} \frac{(\de G) \big(\bar S^{\cal R}_i \big) - (\de G) \big([N] \setminus \bar S^{\cal R}_i \big)}{2} \\
&=  \frac{1}{2N!} \sum_{\cal R} (\de G) \big(\bar S^{\cal R}_i \big) - \frac{1}{2N!} \sum_{\cal R} (\de G) \big([N] \setminus \bar S^{\cal R}_i \big) \\
&=  \frac{1}{2N!} \sum_{\cal R} (\de G) \big(\bar S^{\cal R}_i \big) - \frac{1}{2N!} \sum_{\cal R} (\de G) \big( S^{\cal R}_i \big) \\
&=  \frac{1}{N!} \sum_{\cal R}\frac{ (\de G) \big(\bar S^{\cal R}_i \big) -  (\de G) \big( S^{\cal R}_i \big)}{2}.
\end{align*}
Motivated by this, let us define the coalition game $v = v_G: 2^{[N]} \to \R$
\be\label{value1}
v(S) := \frac{ \de G(S) +\de G([N])}{2} = \frac{ \de G([N]) - \de G([N] \setminus S)}{2}.
\ee
Note that $v(\emptyset) = 0$, $v([N]) = \de G([N])$.  
We may interpret the value function $v(S)$ as the maximal grand coalition value $\de G ([N])$ subtracted by the threat power of the other party $[N] \setminus S$, with the factor of 1/2.

By the fact the value function $v$ is a translation of $\de G / 2$, we see
\be 
  \mathrm{d}_i v (S^{\cal R}_i) = v (\bar S^{\cal R}_i) - v (S^{\cal R}_i) = \frac{ (\de G) \big(\bar S^{\cal R}_i \big) -  (\de G) \big( S^{\cal R}_i \big)}{2} \nn
  \ee
(recall \eqref{gradient}---\eqref{ls}), yielding an alternate expression of allocation
\[
\ga_i G = \frac{1}{N!} \sum_{\cal R}  \mathrm{d}_i v (S^{\cal R}_i).
\]
Notice this is the Shapley value \eqref{eqn:shapleyPermutation} for the coalition game $v$. 
We recall \citet{StTe2019} defined the component game $v_i$ for each $i \in [N]$ as the unique solution in $\cG(2^{[N]})$ to the Poisson's equation $
\mathrm{d} ^\ast \mathrm{d} v_i = \mathrm{d} ^\ast \mathrm{d}_i v$, 
and showed that the component game value at the grand coalition coincides with the Shapley value, that is, $v_i([N]) = \ga_i G$ in this case. Now Theorem \ref{main1} allows us to conclude the following.
\begin{theorem}[Dynamic extension of Nash's and Kohlberg and Neyman's value]\label{sub1}
For a given strategic game $G \in \mathbb G ([N])$, let $v \in \cG(2^{[N]})$ be the coalition game defined as in \eqref{value1}. Let the hypercube graph \eqref{oldG} be equipped with constant weight $\la \equiv 1$, and let $(X_n)_{n \in \N_0}$ be the canonical Markov chain \eqref{MC} with $X_0 = \emptyset$. Then for each player $i \in [N]$ and every coalition $S \subset [N]$, the value allocation operator 
\be
V_i (S) :=  \int_\Omega   \sum_{n=1}^{\tau_{\emptyset, S}(\omega)} \mathrm{d}_i v  \big(X_{n-1}(\omega), X_n(\omega) \big)  d \cP(\omega) \nn
\ee
extends Nash's and Kohlberg and Neyman's value in the sense that
\be
V_i ([N]) = \ga_i G. \nn
\ee
\end{theorem}
\begin{proof}
\citet{StTe2019} gives $v_i([N]) = \ga_i G$. Theorem \ref{main1} gives $v_i = V_i$ on $2^{[N]}$ for all $i \in[N]$, as the Poisson's equation yields a unique solution given the same initial condition $v_i(\emptyset) = V_i(\emptyset) = 0$.
\end{proof}
We refer to \citet{KoNe2021} for a nice review of the historical development of the ideas around the notion of value, as well as several applications to various economic models (also see \cite{KoNe2021-1}).
\begin{remark} \citet{KoNe2021} also introduces the notion of {\em Bayesian games}, which is a game of incomplete information in the sense that the players do not know the true payoff functions, but only receives a signal which is correlated with  the payoff functions; see \cite{KoNe2021} for detailed setup. However, the power of threat, $\de_B G(S)$, of a coalition $S$ in Bayesian game $G$ is still antisymmetric ( $\de_B G(S)= -\de_B G([N] \setminus S)$), and the value allocation also satisfies the representation formula \eqref{formula}. Thus we can conclude  the value of the Bayesian games still admits the stochastic path-integral extension for subcoalitions as in Theorem \ref{sub1}.
\end{remark}

\section{Further examples, beyond coalition games}\label{examples}
In this section we shall present more examples, where in particular, the last example describes a problem in financial decision making and goes beyond the coalitional game setup. Let us begin by revisiting the famous {\em glove game} and the classical Shapley value,  quoted from \cite{StTe2019}.
\begin{example}[Glove game]
  \label{ex:introGlove}
Let $ N = 3 $, and suppose that player
  $1$ has a left-hand glove, while players $2$ and $3$ each have a
  right-hand glove. The players wish to put together a pair of gloves,
  which can be sold for value $1$, while unpaired gloves have no
  value. That is, $ v (S) = 1 $ if $S \subset N $ contains both a left
  and a right glove (i.e., player $1$ and at least one of players $2$
  or $3$) and $ v (S) = 0 $ otherwise. The Shapley values  are
  \begin{equation*}
    \phi _1 (v) = \frac{ 2 }{ 3 } , \qquad \phi _2 (v) = \phi _3 (v) = \frac{ 1 }{ 6 } .
  \end{equation*}
  This is easily seen from \eqref{eqn:shapleyPermutation}: player $1$ contributes
  marginal value $0$ when joining the coalition first (2 of 6
  permutations) and marginal value $1$ otherwise (4 of 6 permutations)
  , so $ \phi _1 (v) = \frac{ 2 }{ 3 } $. Efficiency and
  symmetry then yield
  $ \phi _2 (v) = \phi _3 (v) = \frac{ 1 }{ 6 } $.
\end{example}
We present some new examples henceforth.
\begin{example}[Glove game on an extended graph]
  \label{glovecomplete}
In Shapley's classical coalition game setup, at each stage only one player can join the current coalition, and moreover no one can leave, as can be seen in the Shapley formula \eqref{eqn:shapleyPermutation} and the corresponding hypercube graph \eqref{oldG}.
 
Now our general setup can free up these constraints. For example again let $ N = 3 $, consider the same value function $v$ for the glove game, but now let the game graph $G=(V,E)$ be e.g. such that $V= 2^{[3]}$, and $(S,T) \in E$ iff $S \subsetneq T$. 
Thus in this setup multiple players can join or leave the coalition simultaneously, e.g., from $\{1\}$ to $\{1,2,3\}$ and conversely. But then what is the ``$\mathrm{d}_i v$", the individual contribution for such a state transition?  \cite{StTe2019} sets this as in \eqref{oldpartial}, which looks natural for the hypercube graph. But our framework allows for a complete freedom in the choice of $f_i = \mathrm{d}_i v$. Here, for instance, for $S \subsetneq T$, we may set  
 \begin{equation}\label{newpartial}
    \mathrm{d} _i v \bigl( S , T \bigr) :=
    \begin{cases}
     \frac{1}{|T|-|S|}\big( v(T) - v(S) \big) & \text{if } \ i \in T \setminus S, \\
      0 & \text{if } \ i \in S
    \end{cases}
  \end{equation} 
with the usual convention $  \mathrm{d} _i v \bigl( T, S \bigr) = -   \mathrm{d} _i v \bigl( S , T \bigr)$.   Thus, in each transition, the surplus $\mathrm{d}v(S,T)=v(T) - v(S)$ is equally distributed to the newly incorporated players under this choice of $\mathrm{d} _i v$.  
\end{example}

\begin{example}[Research paper writing game] 
  \label{research}
We want to free up still another restriction in the classical cooperative game setup, namely, the state space for the game needs to be the {\em coalition space} $2^{[N]}$. Instead, in our setup, we can consider a general {\em cooperation state space} $\Xi$, which does not have to be related with the set of players $[N]$. 

To give an example, let $\Xi$ describe the research progress state space on which the game (reward) $v : \Xi \to \R$ is assigned, with the initial state $\emptyset \in \Xi$ and the research completion state $F \in \Xi$. Let $(G, \la)$ be a given game graph with vertices in $\Xi$. Now we define the players contribution measure $f_i =  \mathrm{d}_i v$, for each edge $(S,T) \in E$, by
\be\label{researchpayoff}
\mathrm{d}_i v (S,T)= \frac{1}{N}  \big( v(T) - v(S) \big).
\ee
Thus, unconditionally, the surplus $\mathrm{d}v(S,T)$ is equally distributed to all players involved in this game. Since the path integral of a gradient ($\mathrm{d}v$) depends only on the initial and terminal states, this clearly implies
\be
V_{\mathrm{d}v} (S) = v(S), \text{ and hence } \  V_i (S) = \frac{v(S)}{N}  \ \text{ for all } \  i \in [N] \text{ and } S \in \Xi. \nn
\ee
In particular $ V_i (F) = v(F) / N$, but not only that, for any research progression path $\omega \in \Omega$ towards $F$, \eqref{researchpayoff} clearly yields
\be
\sum_{n=1}^{\tau_{\emptyset, F}(\omega)} \mathrm{d}_i v  \big(X_{n-1}(\omega), X_n(\omega) \big) =  \frac{v(F)}{N} \nn
\ee
implying that the reward is deterministic and not stochastic.
\end{example}
Lastly, we present an application in financial decision problem.
\begin{example}[Entrepreneur's revenue problem] 
  \label{revenue}
In this example let $\Xi$ be the project state space, in which the manager wants to reach the project completion state $F \in \Xi$. The game value $ v(U) $ is the manager's revenue if the project ends up in the state $U$. However at each transition from $S$ to $T$, the manager has to pay $f_i (S,T)$ to the employee $i$, since it is her contribution and share. Thus, in this single transition, the manager’s surplus is $v(T) - v(S) - \sum_i f_i (S,T)$, which can be positive or negative. Thus we do not impose the efficiency condition $\mathrm{d}v = \sum_i f_i$ here, thereby freeing up still another restriction. Moreover, notice that $f_i$ needs not take the form $\mathrm{d}_i v$, i.e., $f_i$ needs not depend on the game $v$.

Now the manager’s revenue problem is, when they start at the initial project state (say $\emptyset$) and if the manager’s goal is reaching the project completion state $F$, what is the expected revenue for the manager?

Observe the answer is $v(F) - \sum_i V_i(F)$, where $V_i$ is defined by the stochastic integral given contribution measures $(f_i)_i$ as in \eqref{value}. (So if this is negative, the manager may not want to start the project at all.)

Moreover the manager may want to recalculate her expected gain or loss in the middle of the project progress. That is, say the current project status is $T$, and they have come to $T$ from $\emptyset$ through a certain path $\omega$, and thus the manager has paid the payoffs -- the path integrals -- \eqref{pathintegral} to the employees. Now the manager may want to calculate the expected gain if she decides to further go on from $T$ to $F$. This is now given by
\[
v(F) - v(T) - \sum_i V^T_i (F),
\]
and the manager can make decisions based on these expected revenue information. And for this, Theorem \ref{main1} shows the stochastic integral $V^T_i$ can be evaluated by solving the equation \eqref{ls2}, and vice versa.
\end{example}

\section{Conclusion} In this paper we reviewed the cooperative game framework of Shapley \cite{Shapley1953a, Shapley1953} and its Hodge-theoretic extension by \citet{StTe2019} and \citet{Lim2021}. These papers regard the cooperative games as value functions on $2^{[N]}$, and \cite{StTe2019, Lim2021} apply the differential operators $\mathrm{d}, \mathrm{d}_i$ defined on the hypercube graph \eqref{oldG}. Then we proposed that the cooperative games may be defined in a  much more general framework of arbitrary weighted game graphs $G=(\Xi, E)$, in which the partial differential $\mathrm{d}_i$ can be replaced by a general contribution measure $f_i \in \ell^2 (E)$. Given $f_i$ for each player $i$, we proposed a natural value allocation operator $V_i$ given by a stochastic path integral driven by the canonical, reversible Markov chain on each weighted graph. Then in Theorem \ref{main1}, we verified an intriguing connection of this stochastic integral with the component game $v_i$, which is the  solution to the Poisson's equation \eqref{ls1}, inspired by the Hodge decomposition \eqref{Hodge}. Now if the efficiency condition  $\sum_i f_i = \mathrm{d}v$ holds for a given cooperative game $v$, then in view of Proposition \ref{facts}, $V_i=v_i$ may be interpreted as a fair and efficient allocation of the cooperation value $v(S)$ to the player $i$ at the cooperation state $S$, which may read as a generalization of the Shapley's allocation formula \eqref{eqn:shapleyPermutation}. However, as illustrated in Example \ref{revenue}, freeing up the efficiency condition allows us to cover even broader range of problems in economics, finance and other social and physical sciences. Finally, in Section \ref{Nash} we explained how our allocation operator $V_i$ can provide a dynamic interpretation and extension of Nash’s and Kohlberg and Neyman’s solution concept for cooperative strategic games.


\end{document}